\documentclass[12pt]{amsart}
\usepackage{amsmath,amssymb}
\usepackage{latexsym}
\usepackage{amsthm}
\usepackage[normalem]{ulem}
\usepackage{lineno}
\usepackage{paralist}
\usepackage{color}

\title[Generically Multiply Transitive Actions]{Groups Acting Generically Multiply Transitively on Solvable Groups}
\date{April 22, 2024}
\author{Ay\c{s}e Berkman}
\address{Mathematics Department, Mimar Sinan Fine Arts University, Istanbul, Turkey}
\email{ayse.berkman@msgsu.edu.tr, ayseberkman@gmail.com}
\author{Alexandre Borovik}
\address{Department of Mathematics, University of Manchester, UK}
\email{alexandre $\gg {\rm at} \ll $ borovik.net}
\thanks{\copyright\ 2024 Ay\c{s}e Berkman and Alexandre Borovik}

\date{\today}

\begin{document}
\newtheorem{problem}{Problem}
\newtheorem{lemma}{Lemma}[section]
\newtheorem{theorem}{Theorem}
\newtheorem{corollary}{Corollary}
\newtheorem*{corollary3}{Corollary 3}
\newtheorem{proposition}[lemma]{Proposition}
\newtheorem{remark}[lemma]{Remark}
\newtheorem{fact}[lemma]{Fact}
\newtheorem{question}[lemma]{Question}
\theoremstyle{definition}
\newtheorem*{definition}{Definition}
\newtheorem{conjecture}{Conjecture}
\newtheorem*{examples}{Examples}
\newcommand{\acf}{algebraically closed field }
\newcommand{\acfd}{algebraically closed field}
\newcommand{\acfs}{algebraically closed fields}
\newcommand{\fmr}{finite Morley rank }
\newcommand{\rk}{\operatorname{rk}}
\newcommand{\ddeg}{\operatorname{deg}}
\newcommand{\psrk}{{\rm psrk}}
\newcommand{\stab}{{\rm stab}}
\newcommand{\fmrd}{finite Morley rank}
\newcommand{\bi}{\begin{itemize}}
\newcommand{\ei}{\end{itemize}}
\newcommand{\sll}{\operatorname{SL}}

\begin{abstract}
In this work, we complete the classification of generically multiply transitive actions of groups on solvable groups in the finite Morley rank setting. We prove that if $G$ is a {connected} group of \fmr acting definably, faithfully and generically $m$-transitively on a connected solvable group $V$ of \fmr where $\rk(V)\leqslant m$, then $\rk(V)=m$, $V$ is a vector space of dimension $m$ over an \acf $F$, $G\cong \operatorname{GL}_m(F)$, and the action is equivalent to the natural action of $\operatorname{GL}_m(F)$ on $F^m$. This generalises our previous work~\cite{bbnotsharp}. As an application of our result, we classify definably primitive groups of \fmr and affine type acting on a set $X$ with a generic transitivity degree of $\rk(X)+1$.
\end{abstract}

\maketitle

\section{Introduction}

It has been known for a long time that sharp $n$-transitivity, for $n\geqslant 4$, is a very restrictive condition on group actions {as the following shows}.

 \begin{fact} {\rm (Jordan, Hall, Tits)}	Let $G$ be a group with a sharply $n$-transitive action for some $n\geqslant 4$, then $G$ is finite and isomorphic to one of the following groups:  ${\rm Sym}(n)$, ${\rm Sym}({n+1})$, ${\rm Alt}({n+2})$,
  $M_{11}$ if $n=4$, or
  $M_{12}$ if $n=5$.
 \end{fact}

\noindent
\textit{Reference.} Jordan proved the result for finite groups \cite{jordan}. Later, Hall \cite{hall} and Tits \cite{tits1951}, independently, generalised this result; in particular, they showed that no infinite group has a sharply $n$-transitive action where $n\geqslant 4$. Even though Tits reference is usually given as \cite{tits1952}, the correct reference seems to be \cite[Chapitre IV]{tits1951}, which can be easily accessed in \cite{titscompleteworks}.

{When the sharpness assumption is removed, then the following is known for finite groups.}

  \begin{fact}	 {Apart from the symmetric and alternating
  		groups, the only finite $4$-transitive groups are the Mathieu groups $M_{11}$, $M_{12}$, $M_{23}$ and $M_{24}$.}
 \end{fact}

\noindent
\textit{Reference.}  {It follows from the classification of finite $2$-transitive groups, which was proved using the classification of finite simple groups;} see \cite[Theorem 4.11]{Cameron1999} and the discussion in \cite[Section 1.6]{smith}.

\medskip

In the above facts, $M_{k}$ denotes the Mathieu groups; they are among the 26 sporadic finite simple groups, their multiply transitive actions are on sets of size $k$. {The maximum transitivity degree for $M_{11}$ and $M_{23}$ is 4; for $M_{12}$ and $M_{24}$ it is 5.}

The following theorem shows that multiply transitive  {rational} actions of algebraic groups are also rare.

  \begin{fact} {\rm \cite{knop}} If a reductive algebraic group acts  {rationally} on an irreducible variety $n$-transitively for some $n\geqslant 2$, then  $n=2$ or $3$. Moreover, the action is equivalent to the action of $\operatorname{PGL}_{m+1}$ on  {the projective space} $\mathcal{P}_m$ for some $m\geqslant 1$, or to the action of $\operatorname{PGL}_{2}$ on  {the projective line} $\mathcal{P}_1$, respectively.
 \end{fact}

A more general notion of transitivity in the class of algebraic actions, introduced by Popov   {\cite{popov}}, {allows} more examples.

\begin{definition} Assume a connected algebraic group $G$ acts  {rationally} on an irreducible  {algebraic} variety $V$.	If the induced action of $G$ on $V^n$ is transitive on an open subset  {of $V^n$}, then the action is called {\em generically} $n$-transitive.
	\end{definition}

 \begin{fact}{\rm \cite{popov}} In the above setting, if the characteristic of the underlying field is $0$, among simple algebraic groups, only those of type $A_n$ have generically  {$4$}-transitive or higher actions.  To be more {accurate}, $A_n$ has the maximal generic transitivity degree $n+2$; $E_6$ has $4$; other types have $2$ or $3$.
 \end{fact}

 The notion of genericity already exists in model theory, more precisely in  {the study of} structures of \fmrd. Recall that when $X$ has \fmrd,  a definable subset $A\subseteq X$ is called {\em generic} (sometimes strongly generic) in $X$, if $\rk(X\setminus A)<\rk(X)$. Using this definition, Borovik and Cherlin \cite{borche} extended the definition of Popov to group actions in the \fmr context.

 \begin{definition} Let $G$ be a group of \fmr acting definably on a set $X$ of \fmrd. If the induced action of $G$ on $X^n$ is (sharply) transitive on a generic subset  {of $X^n$}, then we say $G$ acts {\em generically (sharply) $n$-transitively} on $X$.
 \end{definition}

\begin{examples} \cite{borche}
 For every $m\geqslant 1$ and \acf $F$, the natural action of:
\begin{compactenum}
	\item  $(F^*)^m$ as  {a group of diagonal matrices  on the vector space $F^m$} is generically sharply     $m$-transitive.

	\item  {the general linear group} $\operatorname{GL}_m(F)$ on  {the vector space} $F^m$ is generically sharply $m$-transitive.
	
	\item  {the affine general linear group}  {$\operatorname{AGL}_m(F) = F^m \rtimes \operatorname{GL}_m(F)$} on $F^m$ is generically sharply $(m+1)$-transitive.
	
	\item  {The projective general linear group}  $\operatorname{PGL}_{m+1}(F)$ on  {the projective space} $\mathcal{P}_{m}(F)$ is generically sharply $(m+2)$-transitive.
	
\end{compactenum}
\end{examples}

 Also, in the same paper \cite{borche}, Borovik and Cherlin asked some motivating questions about the concept. We aim to prove the theorem below which answers two questions from their paper in a slightly more general setting.

\begin{theorem} \label{main} Let $G$ be a group of \fmr acting definably, faithfully and generically
	 $m$-transitively on a  {connected} {\em solvable} group $V$ of \fmrd, where $m\geqslant \rk(V)$. Then $V$ is abelian, $m=\rk(V)$,  and, for some \acf $F$,  $G\cong \operatorname{GL}_m(F)$, and the action  {$G$ on $V$} is equivalent to the natural action  {$\operatorname{GL}_m(F)$ on $F^m$}.
\end{theorem}

As an immediate consequence of the above theorem, we note that  {generic $m$-transitive and generic sharply $m$-transitive group actions coincide in this context.}

\begin{corollary} Any definable, faithful and generically
	$m$-transitive action of a group of \fmr on a  {connected} solvable group $V$ of \fmr at most $m$ is generically {\em sharply}
	$m$-transitive.
	\end{corollary}

{In \cite{bbgeneric,bbpseudo,bbsharp,bbnotsharp}, we} worked on this problem and obtained a series of results, that culminated in the following  weaker version of Theorem~\ref{main}, where we assumed that $V$ is an elementary abelian $p$-group, and $p$ is an odd prime.

\begin{fact} {\rm \cite[Theorem 1]{bbnotsharp}}
		Let $G$ be a group of \fmrd, $V$  {a connected} elementary abelian $p$-group of Morley rank $n$, and $p$ an odd prime. Assume that $G$ acts on $V$ faithfully, definably and generically $m$-transitively with $m \geqslant n$. Then $m=n$ and there is an \acf $F$ such that $V\cong F^m$, $G\simeq\operatorname{GL}_m(F)$, and the action is the natural action.
		 \label{elementaryabelianV}
\end{fact}

{The above theorem will be used in proving the main theorem of this paper. Since our main theorem assumes that $V$ is solvable,}
one wonders
what happens when $V$ is not solvable. Combining the following result of Alt\i nel and Wiscons with an observation from \cite{bbsharp}, we deduce a partial answer.

\begin{fact} {\rm \cite{aw2022}} \label{awalt} If\/ $V$ is a connected nonsolvable group of \fmr with no involutions on which ${\rm Alt}(m)$ acts definably and faithfully, then\/ $\rk(V)\geqslant m$.
\end{fact}

\begin{corollary} Let $G$ be a connected group of \fmr acting definably, faithfully and generically sharply
	$m$-transitively on a  {connected} non-solvable group $V$ of \fmr with no involutions, where $m\geqslant \rk(V)$. Then $m=\rk(V)$.
\end{corollary}

\begin{proof}  Proceeding as in the proof of Lemma 3.1 in \cite{bbsharp} we obtain a subgroup isomorphic to ${\rm Sym}(m)\ltimes (\mathbb Z_2)^m$ in $G$. Now Fact~\ref{awalt} applies. \end{proof}

As a final step, we apply our theorem to definably primitive groups of \fmr and affine type. All necessary definitions and proofs are given in Section~\ref{primitive}.

\begin{corollary} Assume that $G$ acts on $X$ definably primitively, and generically $(n+1)$-transitively. If the action is of affine type and $\rk(X)\leqslant n$, then $\rk(X)=n$, and $G\curvearrowright X$ is equivalent to the natural action of the affine group $\operatorname{AGL}_n(K) \curvearrowright {\mathbb A}_n(K)$ for some \acf $K$.
\end{corollary}

\section{Useful Facts}

{In this section, we list references for the facts we will need in our proof.}

\begin{fact} {\rm \cite[Lemma 3.1]{bbpseudo}} If\/ $G\ltimes V$ is a connected group, in which $G$ and $V$ are definable subgroups, $V$ is strongly minimal and $C_G(V)=1$, then there exists an \acf $F$ such that $V\cong F^+$, $G\cong F^*$, and the action is the usual multiplication. \label{strmin}
\end{fact}

\begin{proof} This is a special case of  \cite[Proposition 3.12]{poizat}, which eventually follows from Zilber's theorem,  {see \cite[Theorem 9.1]{bn}}.
\end{proof}

The structure of connected group of Morley rank 1 is well known, as the following fact shows. Recall that a torus (that is, a  {definable} divisible abelian group) is called decent, if it is the definable hull of its torsion
part.

\begin{fact} {\rm \cite[Fact 2.3]{bbsharp}} If $G$ is a connected group of Morley rank $1$, then it is abelian. Moreover, one of the following holds: $G$ is an elementary abelian $p$-group,  or torsion-free and divisible, or a decent torus. \label{rank1groups}
\end{fact}

\begin{fact} {\rm \cite[Fact 2.4]{bbsharp}} A decent torus does not admit a non-trivial connected definable automorphism 	group. \label{torusautom}\end{fact}

\begin{fact} {\rm \cite[Proposition 19.2]{humphreys}} A connected nilpotent  {linear} algebraic group  {over an algebraically closed field} can be written as $G_u\times G_s$ where $G_u$ and $G_s$ are  unipotent and semisimple parts of the group, respectively. \label{humphreys-nilpotent} \end{fact}

Macintyre's theorem on abelian groups   {of \fmrd} will be used.

\begin{fact} {\rm \cite[Theorem 6.7]{bn}}
	Let $G$ be an abelian group of \fmrd, then $G$ can be written as $G=B\oplus D$, where $B$ is a subgroup of bounded exponent and $D$ is  {a definable} divisible subgroup.
	\label{macintyre}
\end{fact}

In the above theorem; being divisible, $D$ is necessarily connected. If $G$ is connected, then $B$ is also necessarily connected.

Next we introduce the concept of the generalised Fitting subgroup and state a related fact. A group $G$ of \fmr has a unique maximal normal nilpotent subgroup, called the Fitting subgroup, which is denoted by $F(G)$. The layer of $G$ is defined to be the subgroup of $G$ generated by (definable) subnormal quasisimple subgroups, and it denoted by $E(G)$. Finally the generalised Fitting subgroup of $G$ is defined as $F^*(G)=\langle F(G), E(G)\rangle$.

\begin{fact}{\rm \cite[Section I.7]{abc}} If $G$ is a group of \fmrd, then $F^*(G)=F(G)\ast E(G)$ is definable, $C_G (F^{*}(G))=Z( F(G))$ and
\[
C_G^\circ(F^{*\circ}(G)) \leqslant F(G).
\]
\label{genfrat}
\end{fact}

The below fact from Borovik and Nesin's book describes the automorphism groups of infinite simple algebraic groups in the \fmr setting.

\begin{fact}{\rm \cite[Theorem 8.4]{bn}}	Let $G\rtimes H$ be a group of finite Morley rank, where
	$G$ and $H$ are definable and connected subgroups. Suppose that $G$ is an infinite simple algebraic group over an
	algebraically closed field and $C_H (G) = 1$. In this case,
	$H$ lies in the group ${\rm Inn}(G)$ of inner automorphisms of $G$,  when viewed as a subgroup
	of the group of all automorphisms of $G$. \label{automs}
\end{fact}

A result by Poizat will be useful.

\begin{fact}{\rm \cite[Theorem 1]{poizat-quelques}, \cite[Proposition II.4.4]{abc}}
	Let $G$  be a simple subgroup of $\operatorname{GL}_n(F)$, where $F$ is an \acf of  {positive characteristic}. Assume that $(\operatorname{GL}_n(F), G)$ is a structure of \fmrd. Then $G$ is definably isomorphic to an algebraic group over $F$.
	\label{poizat}
\end{fact}

We will use the following corollary of the classification of simple groups of even type from \cite{abc}. Recall that $O_2(G)$ stands for the maximal normal 2-subgroup in $G$.

\begin{fact}{\rm \cite[Proposition X.2,  {page 502}]{abc}} \label{centralproduct} Let $G$ be a connected group of finite Morley rank
	containing no nontrivial $2$-torus. Then $O_2(G) = O_2^\circ(G)$ and it is a definable
	unipotent subgroup of $G$. The quotient $G/O_2^\circ(G)$ has the form $Q\ast S$ with $S$
	a central product of quasisimple algebraic groups over algebraically closed
	fields of characteristic 2, and $Q$ a connected group without involutions.
\end{fact}

Next, we state Clifford's Theorem.  {Recall that when a group $G$ acts on an another group $V$ without fixing any non-trivial proper subgroup setwise, we call the action \em{irreducible}.}

	\begin{fact}{\rm \cite[Theorem 11.8]{bn}} Let $G$ be a group of \fmr acting definably, faithfully and irreducibly  on an
			abelian group $V$. Let $H$ be a definable connected normal subgroup of $G$ such that $C_V(H)=0$.
			Then the action of $H$ on $V$ is completely reducible; that is, $V$ is a direct sum of finitely many irreducible $H$-modules. Moreover, these irreducible submodules are conjugate to each other under the action of $G$.
\label{clifford}
\end{fact}

The following classical result by Loveys and Wagner will be essential for the special case when $V$ is divisible abelian.

\begin{fact} {\rm \cite[Theorem A.20]{bn}}\label{loveyswagner}
	Let $G$ be an infinite group  {of \fmr} acting on an infinite divisible abelian group $V$  {of \fmr definably}. If the action is faithful and $G$-minimal, then there exists an \acf $F$ of characteristic 0 such that $V$ is a vector space over $F$, $G$ is definably isomorphic to a subgroup $H$ of $\operatorname{GL}(V)$, and the action  {$G$ on $V$} is equivalent to  {the action of $H$ on $V$}.
\end{fact}

{Recall that in the above context, $G$-minimality means that no proper infinite definable subgroup of $V$ is fixed setwise under the action of $G$. }

\begin{fact} {\rm \cite[Fact 3.3]{bbnotsharp}} For any prime $p$, a $p$-torus can  {definably} act on a connected elementary abelian $p$-group only trivially. \label{toralaction}
\end{fact}

{The following fact was stated  in \cite{bbnotsharp} when $V$ is an elementary abelian $p$-group and $p$ is an odd prime. However, since the proof uses only the commutativity of $V$, we state it in this more general form below.}

\begin{fact} {\rm \cite[Proposition 4.5]{bbnotsharp}} \label{nonconnected} Let $G$ be a group of finite Morley rank  acting definably and faithfully on an abelian group $V$ of Morley rank $n$, which contains a definable subgroup $H \cong \mathop{\rm GL}_n(F)$ for an algebraically closed field $F$. Assume also that $V$ is definably isomorphic to the additive group of the $F$-vector space  $F^n$, and the action of $H$ on $V$ is the natural action. Then $H=G$.\end{fact}

\section{Linearisation Theorems}

We need an expanded version of the following linearisation theorem of Borovik.

\begin{fact}{\rm \cite[Theorem~4]{avb}} \label{linearisation}
Let $p$ be a prime, $F$ an algebraically closed field of characteristic $p$, $V$  {a connected} elementary abelian $p$-group  {of finite Morley rank}, and $G$
a connected algebraic group over $F$, which acts definably, faithfully and irreducibly on $V$.

 Then $V$ has the structure of a finite dimensional vector space over $F$, and the action is of $G$ on $V$ is $F$-linear.
Moreover, the enveloping ring $R(G)$ in this action is the full matrix algebra $\mathop{M}_n(F)$, where $n = \dim_F(V)$.

 \end{fact}

Below is a corollary and a generalisation of the above result.

\begin{theorem}
	{\rm (Linearisation Theorem, Expanded)} \label{cor-linearisation} Let $p$ be a prime, $V \rtimes G$ a group of finite Morley rank, where $V$ is a connected elementary abelian $p$-group of finite Morley rank, and $G$ is a connected group of finite Morley rank which acts on $V$ faithfully, definably, and irreducibly. Assume that $L \lhd G$ is a definable normal subgroup isomorphic to a simple algebraic group over an algebraically closed field $F$ of characteristic $p$.
	
	Then
	\begin{itemize}
		\item The abelian group $V$ has a structure of a finite dimensional   $F$-vector space, and  the action of $G$ on $V$ is $F$-linear.
		\item The groups $G$ can be decomposed as a central product
		\[
		G = L \ast M_1\ast \cdots \ast M_k \ast T
		\]
		where $M_i$, $i = 1, \dots, k$ are definable and isomorphic to simple algebraic groups over $F$ {\rm (}possibly $k=0${\rm )}, and $T=1$ or is a torus.
		\item The enveloping ring $R=R_V(G)$ additively generated by $G$ in $\mathop{{\rm End}} V$ is definable and  equals to the algebra $\mathop{{\rm End}}_F V$.
	\end{itemize}
\end{theorem}

\begin{proof}
	By Fact~\ref{automs}, simple algebraic groups over algebraically closed fields do not allow any non-trivial definable and connected groups of external automorphisms, hence $G = L \ast C_G(L)$. Set $C= C_G(L)$.

By Clifford's Theorem (Fact~\ref{clifford}), $V$ is a direct sum of finitely many irreducible $G$-modules
\[
V =  V_1\oplus \cdots \oplus V_l,
\]
where  $V_i$'s are  irreducible $L$-modules which are conjugate by elements from $G$, and therefore by elements of $C$. Since the groups $L$ and $C$ commute elementwise, subgroups $V_i$ are isomorphic as $L$-modules.

By the above Linearisation Theorem (Fact~\ref{linearisation}), each subgroup $V_i$ has a structure of a vector space over $F$ preserved by the action of $L$ and the enveloping ring $R_i=\mathop{{\rm Env}}_{V_i}(G)$.

 Since all $V_i$'s are isomorphic $L$-modules, their enveloping rings are equal when seen as subrings of $\mathop{{\rm End}}_{V}(G)$.
Hence for all $i$, $R_i = \mathop{{\rm Env}}_{V}(G)$. Let us denote the latter ring by $R$. It is definable by the Linearisation Theorem and isomorphic to the full matrix algebra $M_{d\times d}(F)$, where $d = \dim_F(V_i)$. The center $Z= Z(R)$ could be identified with the field $F$. Since $R$ commutes elementwise with $C$, the field $Z$ also commutes elementwise with $C$ and with the entire group $G$, which makes the action of $G$ on $V$ linear over $Z$ and $F$.

Now we can take care of the structure of the group $G$, using the fact that $G$ is a definable subgroup (but we cannot claim that $G$ is Zariski closed -- see the discussion of this delicate issue in \cite{avb}) of $H = \mathop{\rm GL}_F(V)$.

The group $G$ acts on $V$ faithfully and irreducibly, and for that  reason the unipotent radical of $G$ is trivial; {that is,} $R_u(G)=1$. Consider the generalised Fitting subgroup $F^*(G)= F(G)E(G)$. Notice that, by definition, $L \lhd E(G)$. Let $T$ be the Zariski closure of $F^\circ(G)$, then $T$ is a nilpotent normal subgroup of $G$, and hence the unipotent radical  $R_u(T)$ is also normal in $G$. Thus, $R_u(T) = 1$. Now by Fact~\ref{humphreys-nilpotent}, $T$ is a (decent) torus, and  hence it does not  allow a non-trivial action by a connected group by Fact~\ref{torusautom}.  Thus, $T\leqslant C_H(G)$, and therefore $T\leqslant Z(G)$ and $F^\circ(G) \leqslant Z(G)$.

Now we turn our attention to $E$. We know that $L \lhd E$ and, by properties of $E$,
\[
E = L \ast M_1 \ast \cdots \ast M_k
\]
where $M_i$'s are definable quasisimple normal subgroups  in $G$ (since $G$ is connected). By Fact~\ref{poizat}, all $M_i$'s are simple algebraic groups over $F$.  For convenience, we denote $M_0 = L$, so that
\[
E(G) = M_0 \ast M_1 \ast \cdots \ast M_k.
\]
By the same argument that we applied to $L$ at the beginning of this proof, for all $i = 0,1,\dots, k$,
\[
G= M_i \ast C_G(M_i),
\]
and it follows that $G = E \ast C_G(E)$.

The group $G$ centralises $F^\circ \leqslant Z(G)$, therefore $C_G(E)$ centralises $F^\circ E = F^{*\circ}(G)$.
But $C_G (F^{*\circ}(G)) \leqslant F^{*}(G)$  by Fact~\ref{genfrat}. Hence $G =FE= F^\circ E$, which completes the proof.
\end{proof}

\section{Proof of Theorem \ref{main}}

To be able to prove the main theorem of this work (Theorem~\ref{main}) we have to replace the condition ``$V$ is an elementary abelian $p$-group, where $p$ is odd" with ``$V$ is a solvable group" in the assumptions of Fact~\ref{elementaryabelianV}.

The following simple but useful observations will be used in the sequel without further reference.

\begin{lemma} Assume that $G$ is a group of \fmr acting definably and generically $m$-transitively on a {connected} group $V$ of \fmrd. Then following occur.\\
	{\rm (a)} $\rk(G)\geqslant m\rk(V)$.\\
	{\rm (b)} If $G$ fixes a proper definable normal subgroup $W<V$ setwise, then the action of $G$ on $V/W$ is also generically $m$-transitive.
\end{lemma}

\begin{proof} (a) Let $A$ be the generic orbit in $V^m$, then clearly $\rk(G)\geqslant \rk(A)=\rk(V^m)=m\rk(V)$.\\
	(b) If $G$ fixes $W$, then the projection of the generic orbit in $V^m$ into $(V/W)^m$ is also a generic orbit for the action of $G$.
\end{proof}

 First, we treat the case where $V$ is a  divisible abelian group.

\begin{lemma} Let $G$ be a connected group of \fmr acting definably, faithfully and generically $m$-transitively on a  divisible abelian {connected} group $V$ of \fmrd, where $m\geqslant \rk(V)$. Then $m=\rk(V)$ and the action  {of the group $G$ on $V$} is equivalent to the natural action  {of $\operatorname{GL}_m(F)$ on $F^m$} for some \acf $F$ of characteristic $0$. \label{divisibleV}
\end{lemma}

\begin{proof} We will use induction on $\rk(V)\geqslant 1$. When $\rk(V)=1$, the result follows from Fact~\ref{strmin}. Assume $\rk(V) \geqslant 2$. We will show that the action of $G$ on $V$ is $G$-minimal. Assume $U$ is a proper definable subgroup of $V$ fixed by $G$. Then $G$ acts on $V/U$, which is a divisible abelian group of rank at most $m$, generically $m$-transitively. By the inductive hypothesis, $\rk(V/U)=m$ and thus $\rk(V)=m$ and $\rk(U)=0$. Therefore, Fact~\ref{loveyswagner} applies and there exists an \acf $F$ of characteristic 0 such that $V$ is a vector space over $F$ of Morley rank $m$, $G$ is definably isomorphic to a subgroup $H$ of $\operatorname{GL}(V)$. Note that since $\rk(V)=m$, $\rk(G)=m^2$.
	
	As a last step, we need to show the underlying field $F$ is of Morley rank 1.
		Assume $\rk(F)=k$, and $V$ is a $t$ dimensional vector space over $F$, therefore $m=kt$. Hence $G\cong H\leqslant \operatorname{GL}(V)\cong \operatorname{GL}_t(F)$. Comparing ranks we get $m^2\leqslant t^2\rk(F)=t^2k=m^2/k$. Therefore $k=1$.
\end{proof}

 {Next, we assume that $V$ is an elementary abelian 2-group.}

\begin{proposition} \label{characteristic2} Let $G$ be a connected group of \fmr acting definably, faithfully and generically $m$-transitively on a {connected} elementary abelian $2$-group $V$ of \fmrd, where $m\geqslant \rk(V)$. Then $m=\rk(V)$ and the action of  {the group $G$ on $V$} is equivalent to the natural action  {of $\operatorname{GL}_m(F)$ on $ F^m$} for some \acf $F$ {of characteristic $2$}.
\end{proposition}

\begin{proof} Set $n=\rk(V)$. We will use induction on $n\geqslant 1$. If $n=1$, then Fact~\ref{strmin} applies {as above}. So assume $n\geqslant 2.$
	
First, note that a 2-torus does not have a non-trivial definable action on an elementary abelian 2-group by Fact~\ref{toralaction}, hence $G$ has no non-trivial 2-tori. Therefore, Fact~\ref{centralproduct} applies and we get that $O_2(G)$ is a connected definable unipotent subgroup in $G$. By \cite[Corollary I.8.4]{abc}, $O_2(G)\ltimes V$ is nilpotent. Hence $O_2(G)$ centralises an infinite definable subgroup $W$ of $V$. Thus, $G$ fixes $W$ setwise.
	
However, the same argument as in the proof of Lemma~\ref{divisibleV} shows that the action is $G$-minimal. This contradiction shows that $O_2(G)=1$.
	Hence, by Fact~\ref{centralproduct}, we can write $G=Q\ast S$, where $S=S_1\ast\cdots\ast S_k$ is a central product of quasisimple algebraic groups, and $Q$ is a connected group with no involutions.
Since $G$ contains involutions, $S\neq 1$.

 Set $\bar{V}=V/C_V(G)$, then by \cite[Lemma I.8.1]{abc}, $G$ acts irreducibly on $\bar{V}$. Therefore, Theorem~\ref{cor-linearisation} applies and we  conclude that there exists an \acf $F$ of characteristic 2 such that $\bar{V}$ is a vector space over $F$, $G\leqslant \operatorname{GL}(\bar{V})$ is definable and the action is $F$-linear.
		Since $\rk(G)\geqslant n^2$ and $\rk(\operatorname{GL}(\bar{V}))=n^2$, we get the equality and also conclude that $\rk F=1$.

Next, we will show that $C_V(G)$ is trivial. Let $Z$ be the group of scalars in $\operatorname{GL}(\bar{V})$ then $V=C_V(Z)\oplus [V,Z]$ by  \cite[Corollary I.9.11]{abc}.
			Obviously, $Z$ acts on $\bar{V}$ without fixed points therefore $C_V(Z) = C_V(G)$ and $V=C_V(G)\oplus [V,Z]$.
			Since $C_V(G)$ is finite and $V$ is connected, $C_V(G)=\{0\}$.
	Hence, $G\cong\operatorname{GL}(V)$.
	\end{proof}

Now, we are ready to combine our previous results and treat the case, where $V$ is abelian, without any further restrictions on the structure of $V$.

\begin{proposition} Let $G$ be a connected group of \fmr acting definably, faithfully and generically $m$-transitively on an abelian group $V$ of \fmrd, where $m\geqslant \rk(V)$. Then $m=\rk(V)$ and the action  {of the group $G$ on $V$} is equivalent to the natural action  {of  $\operatorname{GL}_m(F)$ on $F^m$} for some \acf $F$.\label{abelianV}
	 \end{proposition}

\begin{proof} We will use induction on $\rk(V)\geqslant 1$. {When $\rk(V)=1$, Fact~\ref{strmin} applies. Assume $\rk(V)\geqslant 2$.}

By Macintyre's Theorem (Fact~\ref{macintyre}), we can write $V=D\oplus B$, where $D$ is a definable divisible subgroup and $B$ is of bounded exponent. First we will show that $V=D$ or $V=B$. If $V\neq D$ and $V\neq B$, then $G$ acts $m$-transitively on $V/D$, and hence by the induction hypothesis, $\rk(V/D)=m$ and thus $D=0$.
Contradiction shows $V=B$ or $V=D$.
	
		If $V=B$, then let $k$ denote the exponent of $V$. Note that for every divisor $d$ of $k$, $V_d=\{v\in V\mid dv=0\}$ is a definable subgroup of $V$ fixed by $G$. If $V_d$ is a non-trivial proper subgroup then $G$ acts generically $m$-transitively on $V/V_d$, hence each $V_d$ is finite, unless $V_d=V$, that is $k=d$. Therefore, if $k$ has a proper prime divisor $p$ then the homomorphism $V\to V$, $v\mapsto pv$ has a finite kernel, namely $V_p$, and a finite image, namely $V_{k/p}$; hence  $V$ is finite. Contradiction shows that $k=p$ is prime; that is $V$ is an elementary abelian $p$-group. Now the result follows from Fact~\ref{elementaryabelianV} for odd $p$, and from Proposition~\ref{characteristic2} for $p=2$.
	
	If $V=D$, then the result immediately follows from Lemma~\ref{divisibleV}.
		\end{proof}

 {Finally, we are ready to discuss the case where $V$ is solvable.}

\begin{proposition} {Let $G$ be a connected group of \fmr acting definably, faithfully and generically $m$-transitively on a solvable group $V$ of \fmrd, where $m\geqslant \rk(V)$. Then $m=\rk(V)$ and the action  {of the group $G$ on $V$} is equivalent to the natural action  {of $\operatorname{GL}_m(F)$ on $ F^m$} for some \acf $F$.} \label{solvableV}
\end{proposition}

\begin{proof} We will use induction on $\rk(V)\geqslant 1$.
When $\rk(V)=1$, Fact~\ref{strmin} applies. Assume $\rk(V)\geqslant 2$. Then $[V,V]$ is a proper definable connected $G$-invariant subgroup of $V$. Hence, $G$ acts (faithfully and)  generically $m$-transitively on $V/[V,V]$ which is a solvable group of rank at most $\rk(V)$. {Hence,} by the induction hypothesis, $\rk(V/[V,V])=m$ and thus $\rk(V)=m$ and $[V,V]=0$. Therefore, $V$ is abelian, and the result follows from Proposition \ref{abelianV}.
\end{proof}

{The final step to} complete the proof of Theorem~\ref{main} is to remove the connectedness assumption on $G$ from the statement of the above {proposition}. If $G$ is not connected then under the assumptions of Theorem~\ref{main}, $G^\circ$ satisfies all the conditions of {Proposition}~\ref{solvableV} by  \cite[Lemma 4.10]{altwis}. Hence we can conclude that $G^\circ=\operatorname{GL}_m(F)$ and $V=F^m$, for some \acf $F$.
Finally, by Fact~\ref{nonconnected}, we obtain $G=G^\circ$. \hfill $\Box$

This completes the proof of Theorem~\ref{main}.

\section{Primitive Groups of Affine Type}
\label{primitive}

In this section, assume that $G$ is a group of \fmr acting on an infinite set $X$ which is also of \fmr definably and faithfully.

Following \cite{borche}, we will call a definable action $G\curvearrowright X$ {\em definably primitive}, if the only
definable $G$-invariant equivalence relations on $X$ are the trivial relations.

{An important theorem by Macpherson and Pillay \cite[Theorem 1.1]{macpill} divides definably primitive groups of \fmr in four different types of quite different nature.} We will only mention two of these types in this work.

The first type of definably primitive groups that we will consider are  those  groups $G$ which contain a definable normal abelian regular subgroup. In this case, we will say $G$ is of {\em affine type}. We will denote the  definable abelian  normal  regular subgroup of $G$ by $S$ from now on. Moreover, Macpherson and Pillay showed the following.

\begin{fact} \cite{macpill} Assume that $G$ acts on $X$ definably primitively, and is of affine type. Then $S$ is either an elementary abelian p-group for some prime $p$, or is torsion-free divisible. Moreover,  $G=S\rtimes \stab(x)$ for any $x\in X$, and $G\curvearrowright X$ is equivalent to $G \curvearrowright S$. \label{macpillaffine}
\end{fact}

Now we apply our main theorem to definably primitive groups of affine type. Below, we denote the affine group of dimension $n$ over $K$ by $\operatorname{AGL}_n(K)$, which is isomorphic to $\operatorname{GL}_n(K)\ltimes K^n$, and the affine space of dimension $n$ by ${\mathbb A}_n(K)$.

\begin{corollary3} Assume that $G$ acts on $X$ definably primitively, and generically $(n+1)$-transitively. If the action is of affine type and $\rk(X)\leqslant n$, then $\rk(X)=n$, and $G\curvearrowright X$ is equivalent to the natural action of the affine group $\operatorname{AGL}_n(K) \curvearrowright {\mathbb A}_n(K)$ for some \acf $K$.
\end{corollary3}

\begin{proof} Let $S$ be the definable abelian  normal  regular subgroup of $G$. By Fact~\ref{macpillaffine}, we can work with $S$ instead of $X$. The action of $G/S\cong \stab(x)$ on $S$ is generically $n$-transitive, and respects the group structure on $S$.  Now the result follows from Theorem~\ref{main}.
	\end{proof}

Note that there are three infinite families of group actions which are definably primitive of affine type and generically multiply transitive; namely the natural actions of the general linear group, the affine general linear group and the projective general linear group. We claim that these are the only actions with named properties.
	
We state our formal conjecture below.

\begin{conjecture} Assume that $G$ is a connected group of \fmr acting  definably primitively of affine type, and generically $t$-transitively on $X$, where $\rk(X)=n\leqslant t$. Then the action is generically {\em sharply} $t$-transitive (in particular $\rk(G)=nt$), and one of the following occurs.
	\begin{enumerate}
		\item $\rk(G)=n^2$, $t=n$, and the action is equivalent to $\operatorname{GL}_n(K) \curvearrowright K^n\setminus\{0\}$ for some \acf $K$.

		\item $\rk(G)=n(n+1)$, $t=n+1$, and the action is equivalent to $\operatorname{AGL}_n(K) \curvearrowright {\mathbb A}_n(K)$ for some \acf $K$.
		\item $\rk(G)=n(n+2)$, $t=n+2$, and the action is equivalent to $\operatorname{PGL}_{n+1}(K) \curvearrowright {\mathbb P}_n(K)$ for some \acf $K$.
		\end{enumerate}
	\end{conjecture}

The second type of groups in the Macpherson--Pillay classification is definably primitive groups {which contain definable normal  simple regular groups}. We will say that these groups are of {\em regular type}. Here is our conjecture about them.

\begin{conjecture} A connected group of \fmr does not have a definably primitive action of regular type which is also generically $t$-transitive on a set $X$
{with $t \geqslant \rk(X)$}.
\end{conjecture}
{We believe that confirming this conjecture is a feasible task.}

Actually, we can suggest even a much stronger conjecture.

\begin{conjecture} A connected group of \fmr does not have a definably primitive action of regular type which is also generically $2$-transitive.
\end{conjecture}


\begin{thebibliography}{99}
	
	
	
	\bibitem{abc} Alt\i nel, Tuna; Borovik, Alexandre V.; Cherlin, Gregory; {\bf Simple groups of finite Morley rank.} Mathematical Surveys and Monographs, 145. American Mathematical Society, Providence, RI, 2008.
	
	
	
	\bibitem{altwis}  Alt\i nel, Tuna; Wiscons, Joshua; Recognizing PGL$_3$ via generic 4-transitivity. J. Eur. Math. Soc. 20 (2018), 1525--1559.
	
	\bibitem{altwis2}
	Alt\i nel, Tuna; Wiscons, Joshua; Towards the recognition of  $\operatorname{PGL}_n$ via a high degree of generic transitivity. Comm. Algebra. 47 (2019), 206--215.
	
	\bibitem{aw2022} Alt\i nel, Tuna; Wiscons, Joshua; Actions of ${\rm Alt}(n)$ on groups of finite Morley rank without involutions, (2022), arXiv:2205.06173.
	
{\bibitem{bbgeneric} Berkman, Ay\c{s}e; Borovik, Alexandre; A generic identification theorem for groups of finite Morley rank, revisited, (2011)
	arXiv:1111.6037v1.}
	
	\bibitem{bbpseudo} Berkman, Ay\c{s}e; Borovik, Alexandre; Groups of finite Morley rank with a pseudoreflection action. J. Algebra 368 (2012), 237--250.
	
	\bibitem{bbsharp} Berkman, Ay\c{s}e; Borovik, Alexandre; Groups of finite Morley rank with a generically sharply multiply transitive action. J. Algebra 513 (2018), 113--132.
	
	\bibitem{bbnotsharp} Berkman, Ay\c{s}e; Borovik, Alexandre; Groups of finite Morley rank with a generically multiply transitive action on an abelian group, {Model Theory~1 (2022), 1--14.}
	
		
	\bibitem{avb} Borovik, Alexandre; Finite group actions on abelian groups of finite Morley rank, (2020), revised 2023, \textit{to appear in Model Theory, Special Issue for Boris Zilber's 75th}, arXiv:2008.00604v3.
	
	\bibitem{borche} Borovik, Alexandre; Cherlin, Gregory;
	Permutation groups of finite Morley rank. In \textbf{Model theory with applications to algebra and analysis}. Vol. 2, 59--124,
	London Math. Soc. Lecture Note Ser., 350, Cambridge Univ. Press, Cambridge, 2008.
	
	\bibitem{bordel} Borovik, Alexandre; Deloro, Adrien; Binding groups, permutations groups and modules of finite Morley rank, (2019), arXiv:1909.02813.
	
	\bibitem{bn} Borovik, Alexandre; Nesin Ali; {\bf Groups of Finite Morley Rank.} Oxford Logic Guides, 26. Oxford Science Publications. The Clarendon Press, Oxford University Press, New York, 1994.

\bibitem{Cameron1999}  {Cameron, Peter; \textbf{Permutation Groups.} Cambridge University Press, 1999.}
	
	\bibitem{hall} Hall, Marshall, Jr.;
	On a theorem of Jordan,
	Pacific J. Math. 4 (1954), 219--226.
	\bibitem{jordan} 	Jordan, Camille; Recherches sur les substitutions,
Journal de Mathematiques Pures et Appliquees 17 (1872), 351--367.
	
	\bibitem{humphreys} Humphreys, James E.; \textbf{Linear Algebraic Groups}, Graduate Texts in Mathematics, 21. Springer. Corrected fifth printing, 1998.
	
	\bibitem{knop} Knop, Friedrich;
	Mehrfach transitive Operationen algebraischer Gruppen. (German) [Multiply transitive operations of algebraic groups]
	Arch. Math.  41 (1983), no. 5, 438--446.
	
	\bibitem{macpill} Macpherson, Dugald; Pillay, Anand;
	Primitive permutation groups of finite Morley rank,
	Proc. Lond. Math. Soc. 70 (1995), 481--504.
	
	\bibitem{poizat-quelques} Poizat, Bruno; Quelques modestes remarques \`a\ propos
	d'une cons\'equence inattendue d'un r\'esult surpenant de Monsieur
	Frank Olaf Wagner, J. Symbolic Logic 66 no.~4 (2001), 1637--1646.
	
	\bibitem{poizat} Poizat, Bruno; {\bf Stable Groups}, Mathematical Surveys and Monographs, 87. American Mathematical Society, 2001. xiv+129 pp.
	
	 \bibitem{popov} Popov, Vladimir L.; Generically multiply transitive algebraic group actions.  Algebraic groups and homogeneous spaces, 481--523,
	 	Tata Inst. Fund. Res. Stud. Math., 19, Tata Inst. Fund. Res., Mumbai, 2007.
	
	 \bibitem{smith} Smith, Stephen D. {\bf Applying the classification of finite simple groups. A user's guide.} Mathematical Surveys and Monographs, 230. American Mathematical Society, 2018. xiii+231 pp.
	
	 \bibitem{tits1951} Tits, Jacques; Generalisations des groupes projectifs basees sur leurs proprietes de transitivite. (French)
	 Acad. Roy. Belg. Cl. Sci. Mem. Coll. in 8$^\circ$ 27 (1952), no. 2, 115 pp.
	
	 \bibitem{tits1952} Tits, Jacques;
	 Sur les groupes doublement transitifs continus. (French)
	 Comment. Math. Helv. 26 (1952), 203--224.
	
	 \bibitem{titscompleteworks} Tits, Jacques; \OE uvres -- Collected Works, Volume I. Edited by Francis Buekenhout, et al., European Mathematical Society, 2013. xcviii+874~pp.
	
	
	
	
\end{thebibliography}
\end{document}